\documentclass{amsart}

\usepackage[T1]{fontenc}
\usepackage[utf8x]{inputenc}
\usepackage{lmodern, euscript}
\usepackage{enumerate}
\usepackage{mathabx}

\usepackage{amsfonts,amsmath,amstext,amsbsy,amssymb,amsbsy,
  amsopn,amsthm,amscd} 
\usepackage[T1]{fontenc}
\usepackage{hyperref}
\usepackage{mathrsfs}

\RequirePackage[dvipsnames]{xcolor} % [dvipsnames]
\definecolor{halfgray}{gray}{0.55} % chapter numbers will be semi transparent .5 .55 .6 .0
\definecolor{webgreen}{rgb}{0,0.5,0}
\definecolor{webbrown}{rgb}{.6,0,0} \hypersetup{%
  colorlinks=true, linktocpage=true, pdfstartpage=3,
  pdfstartview=FitV,%
  breaklinks=true, pdfpagemode=UseNone, pageanchor=true,
  pdfpagemode=UseOutlines,%
  plainpages=false, bookmarksnumbered, bookmarksopen=true,
  bookmarksopenlevel=1,%
  hypertexnames=true,
  pdfhighlight=/O,%hyperfootnotes=true,%nesting=true,%frenchlinks,%
  urlcolor=webbrown, linkcolor=RoyalBlue,
  citecolor=webgreen, %pagecolor=RoyalBlue,%
  pdftitle={Recap},%
  pdfsubject={},%
  pdfkeywords={},%
  pdfcreator={pdfLaTeX},%
  pdfproducer={LaTeX with hyperref}%
}
\newtheorem{theorem}{Theorem}[section]
\newtheorem{lemma}[theorem]{Lemma}
\newtheorem{corollary}[theorem]{Corollary}
\newtheorem{proposition}[theorem]{Proposition}
\newtheorem{add}[theorem]{Addendum}
\theoremstyle{definition}

\newtheorem{question}[theorem]{Question}

\newtheorem{definition}[theorem]{Definition}

\newtheorem{remark}[theorem]{Remark}

\newcommand{\field}[1]{\mathbb{#1}}
\newcommand{\R}{\field{R}}
\newcommand{\N}{\field{N}}

\newcommand{\Z}{\field{Z}}

\newcommand{\eM}{\EuScript M}
\newcommand{\eS}{\EuScript S}

\newcommand{\cB}{\mathcal B}

\newcommand{\cD}{\mathcal{D}}

\newcommand{\cL}{\mathcal{L}}

\newcommand{\scP}{\mathscr{P}}

\newcommand{\ie}{{\it i.e., }}
\newcommand{\eg}{{\it e.g., }}

\numberwithin{equation}{section}

\renewcommand{\phi}{\varphi}
\renewcommand{\epsilon}{\varepsilon}

\newcommand{\eps}{\varepsilon}

\renewcommand{\|}{\,\Vert\,}

\begin{document}
\baselineskip=14pt

\title{Abelian Livshits theorems and geometric applications}
\author {Andrey Gogolev and Federico Rodriguez Hertz}\thanks{The authors were partially supported by NSF grants DMS-1823150 and DMS-1500947 \& DMS-1900778, respectively}
\dedicatory{dedicated to the memory of Anatole Katok, our mentor and friend}

 \address{Department of Mathematics, The Ohio State University,  Columbus, OH 43210, USA}
\email{gogolyev.1@osu.edu}

\address{Department of Mathematics, The Pennsylvania State University, 
University Park, PA 16802, USA}
\email{hertz@math.psu.edu}

\begin{abstract} 
  \begin{sloppypar}
We introduce a notion of abelian cohomology in the context of smooth flows. This is an equivalence relation which is weaker than the standard cohomology equivalence relation for flows. We develop Livshits theory for abelian cohomology over transitive Anosov flows. In particular, we prove an abelian Livshits theorem for homologically full Anosov flows. Then we apply this theorem to strengthen marked length spectrum rigidity for negatively curved surfaces. We also present an application to rigidity of contact Anosov flows. Some new results on homologically full Anosov flows are also given.
  \end{sloppypar}
\end{abstract}
\maketitle

%\tableofcontents
%KEY WORDS

% Anosov flow, homologically full, cohomology, coboundary, abelian cohomology, abelian coboundary, Livshits theorem, marked length spectrum, Reeb flow, contact Anosov flow, periodic cycle functional, reparametrization, orbit equivalence

\section{Introduction}

Cohomological equations play a crucial role in dynamical systems theory. In the setting when the dynamical system is a smooth flow $X^t\colon M\to M$ generated by the vector field $X$ the equation
\begin{equation*}
\label{eq1}
\phi=L_Xu
\end{equation*}
is called {\it cohomological equation}. Here $u\colon M\to\R$ is an unknown function an $L_X$ is the Lie derivative in the direction of $X$. If the above equation has a solution then the function $\phi\colon M\to\R$ is called an {\it $X^t$-coboundary.}\footnote{We use the term $X^t$-cohomologous rather than just ``cohomologous'' to avoid confusion when we consider cohomology of the manifold.}
Accordingly, two continuous functions $\phi$ and $\psi$ are said to be {\it $X^t$-cohomologous} if $\phi-\psi$ is a coboundary. 
We recommend A. Katok's book for a general introduction of cohomology for dynamical systems~\cite{K}.

The basic question is to decide whether the space of $X^t$-coboundaries is $\cB_X$ is closed in appropriate topology (H\"older, smooth, etc.). If this is the case then $\cB_X$ is given by the intersection of kernels of all $X^t$-invariant bounded linear functionals. Even better is to explicitly describe a family of functionals which would yield $\cB_X$. If $X^t\colon M\to M$ is a transitive Anosov flow on a compact manifold $M$ then this problem was solved by A. Livshits in the classical paper~\cite{Liv}. He proved that in this case the space $\cB_X$, as a subspace of the space of H\"older continuous functions is characterized by countably many functionals which are given by integration over periodic orbits of $X^t$.

Much later, A. Katok and A. Kononenko introduced a new type of functionals called periodic cycle functionals to handle the case of partially hyperbolic dynamical systems. They proved that the space of coboundaries is closed for partially hyperbolic diffeomorphisms and flows which satisfy certain local accessibility condition on its stable and unstable foliation~\cite{KK}. The local accessibility was later relaxed to a more general accessibility condition by Wilkinson~\cite{W}.

In this paper we revisit the idea of Katok-Kononenko and define a version of periodic cycle functionals. We prove several versions of an {\it abelian Livshits Theorem} for transitive Anosov flows, where abelian cohomology is a certain weaker equivalence relation than cohomology.\footnote{To avoid confusion we stress that the term ``abelian" does not refer to the group of values. All cocycles in this paper are $\R$-valued.}  Roughly speaking, this relation is given by the standard cohomology on the universal abelian cover. Then we give an application to marked length spectrum rigidity. More specifically, we improve the Croke-Otal marked length spectrum rigidity theorem (\cite{Cr, Ot}) in the following way. If, for a given pair of negatively curved compact surfaces, the marked lengths spectra on homologically trivial orbits coincide then, in fact, the full marked length spectra match and, by the Croke-Otal~\cite{Cr, Ot} rigidity theorem, the surfaces are isometric. We also apply our machinery to rigidity problem of contact Anosov flows. Namely, under a technical assumption, if two contact Anosov flows are smoothly orbit equivalent then, in fact, the first flow is conjugate to a very special reparametrization of the second flow. Both of these applications also strongly rely on R. Sharp's work on homologically full Anosov flows~\cite{Sh}.

\subsection*{Organization}
In the next section we recall the definition of homologically full Anosov flows and recall some results of Sharp. We provide a new characterization of homologically full Anosov flows in terms of transitivity on the universal abelian cover. We also prove that contact Anosov flows are homologically full.
In Section 3 we state and prove two abelian Livshits theorems: one for general Anosov flows in terms of periodic cycle funcitonals and one for homologically full Anosov flows in terms of homologically trivial periodic data. We also give an ``integer Livshits theorem." In Section 4 we establish several properties of reparametrized flows which will be needed for applications. In particular, we study how equilibrium states and Sharp's minimizers for homologically full Anosov flows behave under reparametrization. In Section 5 we apply an abelian Livshits Theorem to the conjugacy problem of homologically full Anosov flows similarly to the way that the classical Livshits Theorem applies to the conjugacy problem of transitive Anosov flows. Then, in Section~6, we use the preceding results to weaken the assumption in the Croke-Otal rigidity theorem to matching of homologically trivial spectra only. In Section~7 we  also give an application to the conjugacy problem for contact Anosov flows. Finally, in Section 8, we present an instructive example showing that our results from Sections 6 and 7 are optimal in certain ways. 
We also pose two open questions, one at the end of Section~2 and one at the end of Section~3.

\subsection*{Notation} We will denote by $[\alpha]$ the homology class of a periodic orbit or, more generally, a loop $\alpha$. Similarly by $[\omega]$ we will denote the cohomology class of a closed 1-form $\omega$.

%===================

% Given a closed orbit $\gamma$ we will denote by $\delta_\gamma$ the probability invariant measure supported on $\gamma$. We also denote by $\int_\gamma\phi$ the integral of a function $\phi$ along $\gamma$ with unit speed parametrization so that
%$$
%\frac{1}{\ell(\gamma)}\int_\gamma\phi=\int\phi d\delta_\gamma,
%$$
%where $\ell(\gamma)$ stands for the length of $\gamma$.

 %Given a smooth closed manifold $M$ we will use small greek letters $\omega, \theta,\ldots$ to denote smooth $1-$forms. We will use $[\omega], [\theta],\ldots$ to denote corresponding de Rham cohomology classes in $H^1(M,\R)$, but occasionally, when no confusion is possible, we abuse notation and simply write $\omega, \theta,\ldots$ for cohomology classes.

\section{Homologically full Anosov flows}

Here we recall some notions and results due to R. Sharp on homologically full Anosov flows~\cite{Sh}. We also establish some results about such flows which will be needed for our abelian Livshits theory and applications, but may have some independent interest. We begin with a general definition based on the idea that ``homologically full" should mean that points can move asymptotically in any direction in homology of the manifold.

Let $M$ be a compact smooth manifold and let $X^t\colon M\to M$ be a smooth flow. Denote by $\eM(X)$ the space of $X^t$-invariant probability measures. Recall that $\eM(X)\neq \varnothing$ by Krylov-Bogolyubov theorem~\cite{KB}. Following Schwartzman~\cite{Sch} define the asymptotic cycle $\rho\colon\eM(X)\to H_1(M,\R)\simeq Hom(H^1(M,\R),\R)$ by $\rho_\mu([\omega])=\int_M\omega(X)d\mu$. Recall that this is indeed well-defined, \ie $\rho_\mu([\omega])$ is independent of the 1-form representative of the cohomology class $[\omega]$. The define {\it the Schwartzman simplex} as the set $\eS(X)=\{\rho_\mu: \mu\in \eM(X)\}\subset H_1(M,\R)$ which is a compact convex set.
\begin{definition} A smooth flow $X^t\colon M\to M$ is called {\it homologically full} if the origin is contained in the interior of $\eS(X)$.
\end{definition}
\begin{definition} Denote $\EuScript V(X)\subset H_1(M,\R)$ the linear span of $\eS(X)$. Then $X^t\colon M\to M$ is called {\it homologically ample} if the origin is contained in the interior of $\eS(X)$ as a subset of $\EuScript V(X)$.
\end{definition}
For example, the geodesic flow on a flat 2-torus is homologically ample, but not homologically full. On the other hand, if $X^t$ is a transitive homologically ample Anosov flow then it is homologically full. Indeed this follows from the fact that homology classes of periodic orbits of $X^t$ span $H_1(M,\Z)$ which is due to Parry and Pollicott~\cite{PP}. (Plante proved the same for volume preserving flows earlier, using a much simpler argument~\cite{Pl}.)

\begin{remark} Let $N$ be a compact Riemannian manifold. If $N\neq \mathbb T^2$ then the bundle map induces the isomorphism $H_1(T^1N,\R)\simeq H_1(N,\R)$. Let $X^t\colon T^1N\to T^1N$ be the geodesic flow. Then  the Schwartzman simplex $\eS(X)\subset H_1(T^1N,\R)$ coincides with the unit ball in $H_1(N,\R)$ with respect to so called {\it stable norm} which is a very popular object of study in geometry~\cite[Th\'eor\`eme 1.3.9]{M}.
\end{remark}

Now we follow Sharp and recall the definition of the $\beta$-functional for Anoov flows~\cite{Sh}.
Define $\beta\colon H^1(M,\R)\to\R$ using the formula 
$$
\beta([\theta])=\sup_{\mu\in \eM(X)}\left\{h_\mu(X)+\int_M\theta(X)d\mu \right\}
$$
The definition is independent of a particular 1-form representative in the cohomology class. Indeed, if $d\alpha$ is an exact form then $\int_Md\alpha(X)d\mu=\int_ML_X\alpha d\mu=0$ and hence $\beta(\theta+d\alpha)=\beta(\theta)$. Sharp proved that $\beta$ is convex.

\begin{theorem}[\cite{Sh}]\label{thm_sharp}
Let  $X^t\colon M\to M$ be a transitive Anosov flow. Then the following statements are equivalent.
\begin{itemize}
\item[(i)] every homology class in $H_1(M,\Z)$ is represented by a periodic orbit of $X^t$;
\item[(ii)] there exists a fully supported measure $\mu\in\eM$ such that $\int_M\omega(X)d\mu=0$ for all closed 1-forms $\omega$;
\item[(iii)] the functional $\beta\colon H^1(M,\R)\to\R$ is bounded below and there exists unique $\xi_X\in H^1(M,\R)$ for which the minimum is attained;
\item[(iv)] the convex hull of the set $\{[\gamma]: \gamma \,\,\,\mbox{is a periodic orbit}\}\subset H_1(M,\R)$ contains the origin in its interior.
\end{itemize}
\end{theorem}
We note right away that in the context of Anosov flows (iv) is equivalent to our definition of homologically full because measures supported on periodic orbits are dense in $\eM(X)$~\cite{Sig}. We call the cohomology class $\xi_X$ from (iii) {\it Sharp's minimizer}. Further, Sharp established an asymptotic formula for the number of periodic orbits in a fixed homology class. When $\xi_X=0$ the flow is ``balanced" in the sense that asymptotic density of periodic orbits in different homology classes is the same.

%==========WRONG REMARK=========== non-symm finsler pert is a counterexample
%\begin{remark} If $X^t$ is transitive homologically full Anosov flow. Then $\eS(X)$ is symmetric. Indeed,  by ~\cite{Sig} again, $\eS(X)$ is given as the closure of the set $\{\rho_\mu: \mu\,\, \mbox{is supported on a periodic orbit}\}$ which is obviosuly symmetric by $(i)$. Hence $\eS(X)$ is a symmetric convex compact set which contains 0 in its interior. Hence $\eS(X)$ determines {\it Schwartzman norm} on the homology group $H_1(M,R)$. This is a generalization of the stable norm.
%\end{remark}

The following result provides yet one more characterization of homologically full Anosov flows. We need it for our abelian Livshits theory, but it might be of some independent interest as well. Let $\hat M$ be {\it the universal abelian cover} of $M$, that is, the cover which corresponds to the commutator subgroup $[\pi_1M,\pi_1M]$; its group of Deck transformation is given by $H_1(M,\Z)$. 

\begin{theorem}\label{thm_ab_transitive}
An Anosov flow $X^t\colon M\to M$ is homologically full if and only if its lift to the universal abelian cover is transitive.
\end{theorem}

\begin{remark} Transitivity on the universal abelian cover is equivalent to absence of wandering points. Indeed, transitivity on the abelian cover clearly implies that the lift has no wandering points. In the other direction, by Anosov closing lemma, full non-wandering set implies that periodic points are dense. Then transitivity follows by the standard Smale's argument as in the proof below.
\end{remark}

\begin{proof} We begin with the following lemma.
\begin{lemma} If $X^t\colon M\to M$ is a homologically full transitive Anosov flow then homologically trivial periodic orbits are dense in $M$.
\label{lemma_dense}
\end{lemma}
	By~\cite[Theorem 2]{Sh}, homologically trivial orbits equidistribute according to an equilibrium measure given by a potential $\xi(X)$, where $\xi$ is a certain closed 1-form. (When $X$ is a geodesic flow $\xi=0$ and homologically trivial orbits equidistribute to the measure of maximal entropy.) In particular, because equilibrium states are fully supported, it follows that the homologically trivial orbits are dense.  However, one can avoid using Sharp's machinery altogether and give a simpler proof by using shadowing. We briefly sketch this proof.
\begin{proof}[Proof of Lemma~\ref{lemma_dense}]
For any $\eps>0$ we will construct an $\eps$-dense periodic orbit. Begin with an $\eps$-dense geodesic $\gamma_\eps$. Then by (i) of Theorem~\ref{thm_sharp} there exist geodesics $\gamma_\eps'$ in the opposite homology class, \ie such that the sum $\gamma_\eps+\gamma_\eps'$ bounds a 2-cycle.

Consider a point $x\in M$ such that $\gamma'(t_\eps)\to x$, $\eps\to 0$. Pick a Markov partition for $X^t$ in such a way that $x$ is in the interior of a Markov rectangle $R$. Both $\gamma_\eps$ and $\gamma_\eps'$ intersect $R$ at two points $p$ and $q$, respectively, which are $\eps$-close to each other. By concatenating symbolic periods of $\gamma_\eps$ and $\gamma_\eps'$ we can find a periodic orbit $\eta_\eps$ which $\eps$-shadows $\gamma_\eps$ first and then $\eps$-shadows $\gamma_\eps'$. Orbit $\eta_\eps$ intersects $R$ very close to $[p,q]$ once and then intersect $R$ very close to $[q,p]$. (Of course there could be more points of intersection with $R$ corresponding to other points of intersection of $\gamma_\eps$ and $\gamma_\eps'$ with $R$.). Applying Fried's construction~\cite[pp. 300-301]{Fri} to $\gamma_\eps$, $\gamma_\eps'$ and $\eta$ yields a 2-dimensional immersed surface whose boundary consists of these periodic orbits;\footnote{Fried considers 3-dimensinal Anosov flows, but this particular construction works well in any dimension.} Moreover, $\eta_\eps$ is homologous to $\gamma_\eps+\gamma_\eps'$ and, hence, is homologically trivial. It remains to notice that $\gamma_\eps$ is contained in the $\eps$-neighborhood of $\eta$. Hence $\eta_\eps$ is $2\eps$-dense in $M$.
\end{proof}

First assume that $X^t$ is homologically full. Denote by $\hat M$ the universal abelian cover of $M$ and by $\hat X^t\colon\hat M\to\hat M$ the lift of the flow $X^t$. Note that homologically trivial periodic orbits in $M$ are precisely those periodic orbits which lift to periodic orbits of $\hat X^t$.  Hence, by the above lemma, periodic orbits of $\hat X^t$ are dense in $\hat M$. Now we can apply a standard argument of Smale~\cite[(7.5)]{smale}  to conclude that $\hat X^t\colon\hat M\to\hat M$ is indeed a transitive flow.
Namely, given open sets $U$ and $V$ one can wait until $U$ returns to itself and then connect this recurrent subset of $U$ to $V$ via a chain of stable and unstable manifolds of periodic point (which are dense by transitivity) and then apply the $\lambda$-lemma to show that some of the points eventually arrive in $V$\footnote{With some more care one could show topological mixing property.}.

It remains to check the converse implication. (This implication is not needed for Theorem~\ref{livsic}.) Assume that $\hat X^t$ is transitive, \ie $\{\hat X^t(\hat x):t\in\R\}$ is dense in $\hat M$ for some $\hat x\in\hat M$. 

The homology group $H_1(M,\Z)$ is identified with the group of Deck transformation of the cover $\hat M\to M$. Take any $\gamma\in H_1(M,\Z)$. Then for some $t_0>0$ the point $\hat X^{t_0}(\hat x)$ is very close to $\gamma(x)$ so that $\hat X^{t_0}(\hat x)$ and $\gamma(\hat x)$ belong to the same small local product structure chart. Denote by $x$ the image of $\hat x$ in $M$. Then $X^{t_0}(x)$ and $x$ also belong to the same small local product structure chart. Hence, by Anosov closing lemma, the orbit segment $[x, X^{t_0}(x)]$ can be shadowed by a periodic orbit of a point $y$, $X^{t_1}(y)=y$, $t_1\approx t_0$, which is very close to $x$. By the shadowing property the orbit of $y$ is homotopic to the orbit segment $[x, X^{t_0}(x)]$ concatenated with a short curve connecting $X^{t_0}(x)$ back to $x$. It follows that $\hat X^{t_1}(\hat y)=\gamma(\hat y)$, where $\hat y$ is a lift of $y$. That is, $\gamma$ is represented by a periodic orbit of $X^t$ and, hence, $X^t$ is homologically full.
\end{proof}

The following theorem provides a natural class of homologically full Anosov flows. It is interesting whether any contact flow can be shown to be homologically ample.

\begin{theorem}
\label{full}
Let $X^t\colon M\to M$ be a contact Anosov flow. Then $X^t$ is homologically full.
\end{theorem}

\begin{proof}
Let $\alpha$ be the positive contact form for $X^t$ and let $m=\alpha\wedge(d\alpha)^k$ be the invariant volume form. Recall that by (ii) of Theorem~\ref{thm_sharp}: {\it $X^t$ is homologically full if for every $[\omega]\in H^1(M,\R)$ }
$$
\int_M\omega(X)m=0
$$

So let $\omega$ be a closed 1-form. Note that $\omega\wedge (d\alpha)^k$ is a top-dimensional form and, hence, $\omega\wedge (d\alpha)^k=\psi \alpha\wedge(d\alpha)^k$. Contraction with $X$ yields $\omega(X)(d\alpha)^k=\psi\iota_X(\alpha\wedge(d\alpha^k))=\psi(d\alpha)^k$. Hence $\psi=\omega(X)$. Now we have
\begin{multline*}
\int_M\omega(X)m=\int_M\omega(X)\alpha\wedge (d\alpha)^k=\int_M\omega\wedge (d\alpha)^k=\int_M-d(\omega\wedge\alpha\wedge(d\alpha)^{k-1})=0
\end{multline*}
\end{proof}

\begin{remark} I was pointed out to us by G. Paternain that the above result is well-known, see \eg~\cite[Corollary 4.10]{Pl2}.
\end{remark}

Asaoka proved that any transitive codimension-1 Anosov flow is orbit equivalent to a volume preserving Anosov flow~\cite{A}. Also recall that Foulon and Hasselblatt developed contact surgery and created many examples of 3-dimensional contact Anosov flows~\cite{FH}.

\begin{question} Is every 3-dimensional homologically full Anosov flow orbit equivalent to a contact Anosov flow?
\end{question}

%===============ABELIAN COHOMOLOGY======================================================

\section{Abelian cohomology for Anosov flows}\label{cohomologicalequationflows}

Let $X^t\colon M\to M$ is a transitive Anosov flow on a closed compact manifold $M$ of arbitrary dimension. A H\"older continuous function $\phi$ is called an {\it abelian coboundary} if there exists a smooth closed $1-$form $\omega$ and a H\"older continuous function $u$, which is continuously differentiable along $X$, such that 
\begin{equation}
\label{eq2}
\phi=\omega(X)+L_Xu
\end{equation}
 Here $\omega(X)$ stands for the contraction given by evaluation of $\omega$ on the generating vector field $X$. Accordingly, we say that two function $\phi$ and $\psi$ are {\it abelian cohomologous} if $\phi-\psi$ is an abelian coboundary.

\begin{remark}
\label{rmk_liv}
Notice that the decomposition~(\ref{eq2}) is highly non-unique because we can change $\omega$ by any exact 1-form. Indeed, given any smooth function $v\colon M\to\R$ we can write a different decomposition
$$
\phi=(\omega+dv)(X)+L_X(u-v)
$$
However one could make some canonical choice for example by asking $\omega$ to be harmonic with respect to a Riemannian metric. (Recall that given a fixed Riemannian metric there exists a unique harmonic representative in each cohomology class.)
\end{remark}
 
 We develop the counterpart of the standard Livshits theory~\cite{Liv} for abelian cohomology. Specifically, we prove two {\it abelian Livshits Theorems} for transitive Anosov flows:
\begin{itemize}
\item General Livshits Theorem~\ref{katokkononenkoabelianthm} which characterizes the space of abelian coboundaries as the intersection of kernels of periodic cycle functionals;
\item Livshits Theorem~\ref{livsic} for homologically full Anosov flows which characterizes abelian coboundaries via obstructions given by integration over homologically trivial periodic orbits;
%\item A Livshits Theorem~\ref{proposition:periods} for cocycles whose periodic orbits obstructions take values to a rank one abelian subgroup of $\R$.
\end{itemize}
 At the end of this section we also give a similar proposition for functions whose periodic orbits obstructions take values in a rank one abelian subgroup of $\R$ and pose an open question for the case of finite rank.

Now we explain the term ``abelian.''
Recall that the universal abelian cover $\hat M\to M$ is the cover which corresponds to the commutator subgroup $[\pi_1M,\pi_1M]$. 
%Let $\hat M$ be {\it the universal abelian cover} of $M$, that is, the cover which corresponds to the commutator subgroup $[\pi_1M,\pi_1M]$; its group of Deck transformation is given by $H_1(M,\Z)$. 
Then the lift $\hat\omega$ of any closed 1-form $\omega$ on $M$ is exact on $\hat M$. Hence a lift $\hat\phi$ to $\hat M$ of an abelian coboundary $\phi\colon M\to\R$ is a true coboundary for the lifted Anosov flow $\hat X$ because
$$
\hat\phi=\hat\omega(\hat X)+L_{\hat X}\hat u=d\alpha(\hat X)+L_{\hat X}\hat u=L_{\hat X}(\alpha+\hat u)
$$

\subsection{Katok-Kononenko theory of periodic cycle functionals revisited}
Given an Anosov flow $X^t\colon M\to M$ a {\it $us$-adapted path} is a piecewise smooth path $\gamma\colon[0,1]\to M$ such that each of its legs lies entirely in a stable or an unstable leaf of $X^t$. Analogously, an {\it $Xus$-adapted path} (or simply an {\it adapted path}) is a piecewise smooth path $\gamma$ such that each of its legs is either a flow-line segment or  lies entirely in a stable or an unstable leaf. An {\it $Xus$-adapted loop} is an $Xus$-adapted path which begins and ends at the same point.

%Define $\gamma_1\ast\gamma_2$. Define equivalence among paths, s-equivalence, u-equivalence and X-equivalence. (follow katok-kononenko, damjanovic-katok and wilkinson).

%There are two definition if $\ast$ operation, one with base point and the other without base point, the second is more useful, and needs the notion of equivalence among paths and loops. 

Given an adapted path or loop $\gamma$ we proceed to define {\it periodic cycle functionals} $PCF_\gamma:C^{\alpha}(M,\R)\to\R$ as follows.
If $\gamma$ lies entirely in a stable leaf then let
$$
PCF_{\gamma}(\phi)=\int_0^{\infty}\phi(X^t(\gamma(0)))-\phi(X^t(\gamma(1)))dt
$$
%\begin{eqnarray*}
%PCF_{\gamma}(\phi)&=&-\int_0^{\infty}\phi(X^t(\gamma(1)))-\phi(X^t(\gamma(0)))dt=-\int_0^{\infty}\left(\int_\gamma d(\phi\circ X^t)\right)dt\\
%&=&	-\int_0^{\infty}\left(\int_\gamma X^t_*(d\phi)\right)dt\\
%\end{eqnarray*}
If $\gamma$ lies entirely in an  unstable leaf then let
$$
PCF_{\gamma}(\phi)=\int_{-\infty}^0\phi(X^t(\gamma(1)))-\phi(X^t(\gamma(0)))dt
$$
%\begin{eqnarray*}
%PCF_{\gamma}(\phi)&=&\int_{-\infty}^0\phi(X^t(\gamma(1)))-\phi(X^t(\gamma(0)))dt=\int_0^{\infty}\left(\int_\gamma d(\phi\circ X^t)\right)dt\\
%&=&\int_0^{\infty}\left(\int_\gamma X^t_*(d\phi)\right)dt\\
%\end{eqnarray*}
Note that convergence follows from exponential contraction/expansion and H\"older continuity of $\phi$. If $\gamma$ is a positively oriented orbit segment $\gamma=[x,\phi^T(x)]$, $T>0$, then let 
$$
PCF_{\gamma}(\phi)=\int_0^T\phi(X^t(x))dt
$$
and if $\gamma=[x,\phi^T(x)]$, $T<0$, then let
$$
PCF_{\gamma}(\phi)=-\int_T^0\phi(X^t(x))dt
$$
Finally for an adapted path $\gamma$ define $PCF_\gamma(\phi)$ as a the sum of values on each of the legs. The following properties are immediate from the definitions.
\begin{enumerate}
\item Any continuous path can be $C^0$ approximated by an $Xus$-adapted path;
\item The value of of $PCF_\gamma(\phi)$ only depends on the sequence of the endpoints of the legs of $\gamma$ and is independent of the choice of leg between the endpoints;
%\item One can finitely subdivide any leg of $\gamma$ without changing the value of $PCF_\gamma(\phi)$;
\item If $\bar\gamma$ denotes the adapted path (or loop) with reversed orientation then $PCF_{\bar\gamma}(\phi)=-PCF_\gamma(\phi)$;
\item Suppose $\alpha$ and $\beta$ are adapted loops such that $\alpha$ contains a subpath $\gamma$ and $\beta$ contains $\bar\gamma$, the same subpath with the opposite orientation. Concatenating  $\alpha$ with $\gamma$ removed and $\beta$ with $\bar\gamma$ removed results in a loop $\alpha*\beta$. Then we have the {\it additive property}
\begin{equation}
\label{eqpcf}
PCF_{\alpha*\beta}(\phi)=PCF_\alpha(\phi)+PCF_\beta(\phi)
\end{equation}
\end{enumerate}

\begin{remark}
As mentioned earlier, originally periodic cycle functionals were introduced by Katok and Kanonenko~\cite{KK} as obstructions, given by adapted $us$-loops, to solving the cohomological equation in the setting of partially hyperbolic diffeomorphisms when periodic orbits obstructions are not readily available. Note that our definition is different from the original one as we allow the flow direction in the definition of the adapted loop.
\end{remark}

If $\phi$ is an $X^t$-coboundary, $\phi=L_Xu$, then it is easy to see that in all three cases ($\gamma$ is contained in a stable leaf, unstable leaf or an orbit segment) we have $PCF_{\gamma}(\phi)=u(\gamma(1))-u(\gamma(0))$. Hence periodic cycle functionals of $Xus$-adapted loops vanish. Similarly if $\phi$ is an abelian coboundary then periodic cycle functions vanish on  homologically trivial $Xus$-adapted loops because these are the loops which can be lifted to the universal abelian cover where $\phi$ becomes a true coboundary. We prove that vanishing on homotopically (and even homologically) trivial $Xus$-adapted loops is also a sufficient condition for being an abelian coboundary.

\begin{theorem}\label{katokkononenkoabelianthm}
	Let $X^t\colon M\to M$ be a transitive Anosov flow and let $\phi\in C^{r}(M)$, $r>0$. Assume that $PCF_{\gamma}(\phi)=0$ for every homotopically trivial $Xus$-adapted loop $\gamma$. Then there exist a smooth closed 1-form $\omega$ and $u\in C^{r_*}$,  such that 
	$$\phi=\omega(X)+L_Xu$$
\end{theorem}	
In the Theorem above, $r_*=r$ if $r\notin \N$ and $r_*=r-1+\tiny{Lip}$ if $r\in\N$.

We immediately obtain the following corollary.
\begin{corollary}
	If  function $\phi:M\to\R$ above is $C^\infty$ smooth then there is a $C^\infty$ smooth closed 1-form $\omega$ such that 
	$$\phi=\omega(X)$$ 
\end{corollary}

\begin{proof}[Proof of Theorem \ref{katokkononenkoabelianthm}]

We lift all the objects to the universal cover $\tilde M$ and, by a light abuse of notation, we still denote by $\phi$ and $X$ the lifts of $\phi$ and $X$ to the universal cover. Pick a point $a\in\tilde M$ and define $\tilde u_{a}:\tilde M\to\R$ in the following way. Given a point $x\in\tilde M$ let $\gamma$ be an adapted path starting at $a$ and ending at $x$ we set 
$$
\tilde u_a(x)=PCF_\gamma(\phi)
$$
 This definition is independent of the choice of $\gamma$ because we have assumed that periodic cycle functionals vanish on homotopically trivial adapted loops. 
 %We can see from this definition that $\tilde u_a$ is $C^r$ along the stable leaves and unstable leaves and that $u$ and is $C^{r+1}$ along the flow. Hence, by Journ\'e's theorem~\cite{journe}, $\tilde u$ is $C^{r-\epsilon}$ and, 
 Note that, because we can assume that the last leg of $\gamma$ is a flow segment, we have $L_X\tilde u_a=\phi$ (though we will not use this last fact).
 It easily follows from the definition that for any pair of points $a,b\in\tilde M$
 \begin{equation}
 \label{u_a}
 \tilde u_a-\tilde u_b=\tilde u_a(b)
 \end{equation}
% and, hence, $\tilde u_a(b)+\tilde u_b(a)=0$.

	Let $\cD\simeq\pi_1M$ be the group of deck transformations acting on $\tilde M$. The function $\tilde u_a$ solves the cohomological equation on $\tilde M$ , but a priori  is not $\cD-$invariant, so it needs to be adjusted.

	Let $T\in \cD$ and let $\gamma$ be an adapted path. Because $\phi\circ T=\phi$ we have 
	$$PCF_\gamma(\phi)=PCF_{T(\gamma)}(\phi)$$ 
	Hence for every $T\in\cD$ 
	$$
	\tilde u_{T(a)}(T(x))=\tilde u_a(x)
	$$
Define $c\colon\cD\to\R$ in the following way
 $$
 \tilde u_{a}(T(x))-\tilde u_a(x)=\tilde u_{a}(T(x))-\tilde u_{T(a)}(T(x))=\tilde u_a(T(a))\stackrel{\mathrm{def}}{=}c(T)\in\R
 $$
 where we have used~(\ref{u_a}).
 Then $c$ is a homomorphism. Indeed,
 $$
 c(T\circ S)=\tilde u_a(T(S(a))=\tilde u_a(T(S(a))-\tilde u_a(S(a))+\tilde u_a(S(a))=c(T)+c(S)
 $$
  Notice also that $c$ does not depend on the choice of the base point $a$.
  
Now we use the isomorphism $Hom(\cD,\R)\simeq H^1(M,\R)$. Recall that the cohomology class corresponding to $c\colon\cD\to\R$ is represented by a closed 1-form $\omega$ such that
   $$
   c(T)=\int_{\gamma_T}\omega(\dot\gamma_T(s))ds
   $$
   where $\gamma_T$ is any curve starting at $x$ and ending at $T(x)$.

Then $\bar\phi=\phi-\omega(X)$ is invariant under the action of $\cD$ and hence descends to a function $\bar\phi\colon M\to\R$. Now take any periodic orbit $\gamma$ in $M$ and lift it to an orbit segment $\tilde\gamma$ in $\tilde M$. Then, of course, $PCF_{\tilde\gamma}(\bar\phi)=PCF_{\gamma}(\bar\phi)$. Let $x$ and $T(x)$, $T\in\cD$, be the endpoints of $\tilde\gamma$. We have

\begin{eqnarray*}
	\int_\gamma\bar\phi&=&\int_{\tilde\gamma}\bar\phi=\int_{\tilde\gamma}\phi-\int_{\tilde\gamma}\omega(X)=\tilde u_{x}(T(x))-\tilde u_x(x)-c(T)\\
	&=&\tilde u_a(T(x))-\tilde u_a(x)-c(T)=0
\end{eqnarray*}
Hence
 $$\int_\gamma\bar\phi=0$$ 
 for every closed orbit $\gamma$. Then by Livshits Theorem~\cite{Liv} there exists a H\"older continuous $u$, continuously differentiable along $X$ such that
 that $L_Xu=\bar\phi$, \ie
 $$
 \phi=L_Xu+\omega(X)
 $$
 Further the H\"older exponent of $u$ is the same as the H\"older exponent for $\phi$. If $\phi$ is $C^r$ with $r>1$ then de la Llave-Marco-Mariy\'on smooth Livshits Theorem~\cite[Appendix A]{LMM} applies and together with Journ\'e's regularity lemma~\cite{journe} yields $C^{r_*}$ regularity of $u$.
\end{proof}

\subsection{Livshits Theorem for homologically trivial orbits.}

\begin{theorem}\label{livsic}
	Assume that $X^t\colon M\to M$ is a homologically full transitive Anosov flow and let $\phi\colon M\to\R$ be a $C^{r}$, $r>0$ function such that
	$$
	\int_\gamma\phi=0
	$$
	 for all homologically trivial closed orbits $\gamma$. Then there is a $C^\infty$ smooth closed 1-form $\omega$ on $M$ and a function $u\in C^{r_*}(M)$  such that 
	 $$
	 \phi=\omega(X)+L_Xu%=\omega(X)+du(X)
	 $$ 
%	If we additionally assume that $\phi$ is smooth then the transfer function $u$ is also smooth.  
\end{theorem}

\begin{remark} Notice that any homologically trivial periodic orbit bounds a $2$-cycle. Hence, using~(\ref{eqpcf}), the integral $\int_\gamma\phi$ can be decomposed into sum of periodic cycle functionals of homotopically trivial $Xus$-adapted loops. Hence, it is easy to see that vanishing of periodic cycle functionals on homotopically trivial $Xus$-adapted loops implies vanishing on all homologically trivial periodic orbits. Thus Theorem~\ref{livsic} can be viewed as a strengthening of Theorem~\ref{katokkononenkoabelianthm} in the setting of homologically full Anosov flows.
\end{remark}

%The following lemma will be used later in the paper.
%\begin{lemma} 
%\label{lemma_positive_negative}
%If $X^t\colon M\to M$ is a homologically full transitive Anosov flow and let $\phi\colon M\to\R$ be an abelian coboundary. Then either $\phi$ takes both positive and negative values or $\phi=0$.
%\end{lemma}
%\begin{proof} 
%Because $\phi$ is an abelian coboundary, we have $\int_\gamma\phi=0$ for every homologically trivial periodic orbit $\gamma$. Hence either $\phi$ takes both positive and negative values or $\phi|_\gamma=0$ for every homologically trivial periodic orbit $\gamma$. But such periodic orbits are dense, hence, in the latter case $\phi=0$.
%\end{proof}

\begin{proof}%[Proof of Theorem~\ref{livsic}]
	
	Let $\hat M$ be  the universal abelian cover of $M$. We denote by $\hat\phi$ the lift of $\phi$ to $\hat M$ and we still write $X^t$ for the lift of the flow as it won't cause any confusion. By Theorem~\ref{thm_ab_transitive} the lifted flow is transitive.
	
	 Let $x\in\hat M$ be a point with a dense orbit. Define 
	$$
	\hat u(X^t(x))=\int_0^t\hat\phi(X^\tau(x))d\tau
	$$ 
	By the classical argument of Livshits, $u$ is H\"older continuous (with a uniform constant) on the orbit of $x$ and hence extends to a H\"older function on $\hat M$. Further $\hat u$ is continuously differentiable along the flow direction and solves the cohomological equation
	 $$
	 L_X\hat u=\hat\phi
	 $$
	
	Let $\cD\simeq H_1(M,\Z)$ be the group of deck transformations of the covering $\hat M\to M$. Because $\hat\phi$ is $\cD$ invariant, we have $L_X(\hat u\circ T-\hat u)=0$ for every $T\in\cD$. And because $X^t$ has a dense orbit we conclude that $\hat u\circ T-\hat u$  is constant. Let $c(T)=\hat u\circ T-\hat u$. Then $c:\cD\to\R$ is a homomorphism, indeed
	$$
	c(T\circ S)=\hat u\circ (T\circ S)-\hat u=\hat u\circ (T\circ S)-\hat u\circ S+\hat u\circ S-\hat u=c(T)+c(S)
	$$ 
Now identify $\cD$ with its orbit in $\hat M$. By the de Rham Theorem we can extend $c:\cD\to\R$ to a smooth function $c\colon\hat M\to \R$ which is equivariant with respect to the $\cD$ action, that is,
$$
c\circ T- c=c(T)
$$
Let $\hat\omega=dc$. Then $\hat \omega$ is an exact 1-form which is invariant under the action of $\cD$. Hence it descends to a closed 1-form $\omega$ on $M$. (Form $\omega$ is a de Rham representative in the cohomology class given by $c\in \textup{Hom}(H_1(M,\R),\R)$.)
Function $\hat u-c$ is $\cD$-invariant and hence descends to a function $u$ on $M$. We have
$$
L_X(\hat u-c)=\hat\phi-\hat \omega(X)
$$
Thus
$$
L_Xu=\phi-\omega(X)
$$
Finally, if $\phi$ is $C^r$ then, as in the proof of Theorem~\ref{katokkononenkoabelianthm},  smooth Livshits theory~\cite[Appendix A]{LMM}, \cite{journe} yields $C^{r_*}$ regularity of $u$.
\end{proof}

%=======================================================================

\subsection{Integer periods Livshits Theorem}

We will prove following proposition using the circle valued Livshits Theorem.

\begin{proposition}\label{proposition:periods}
	Let $X^t\colon M\to M$ be a transitive Anosov flow and let $\phi\colon M\to \R$ be a H\"older  continuous function. Let 
	$$
	\scP=\left\{\int_\gamma\phi:\gamma\in Per(X)\right\}\subset\R
	$$ 
	be the set of periods of $\phi$. If the rank of the additive group generated by $\scP$ is one, then there are a smooth closed $1-$form $\omega$ and a H\"older continuous function $u$ such that $\phi$ is an abelian coboundary, \ie
	$$
	\phi=\omega(X)+L_Xu
	$$
\end{proposition}

Recall that integral cohomology $H^1(M,\Z)$ is torsion-free and can be regarded as Bruschlinsky group of homotopy classes of maps $\{[M\to S^1]\}$. Indeed,  a smooth function $w\colon M\to S^1$ defines a closed integral form $\omega=dw$. The correspondence $[w]\mapsto[\omega]\in H^1(M,\Z)$ is in fact an isomorphism. 

Note that, if $w\colon M\to S^1$ is not smooth then, by Whitney Approximation Theorem it can be approximated by a smooth map $v\colon M\to S^1$. The function $\bar u=w-v$ has its image in a small interval and hence lifts to a function 
$u\colon M\to\R$.

\begin{proof}
	Assume that $\scP\subset c\Z$, $c\neq 0$. By a constant reparametrization we can assume that $c=1$. Consider the cocycle
	$$
	\hat\phi(x, T)=\int_0^T\phi(X^s(x))ds\in\R
	$$
	 and consider $\Phi(x,T)=[\hat\phi(x, T)]\in S^1=\R/\Z$. By assumption, for every  periodic  orbit $p$ of period $T_p$ we have $\Phi(p,T_p)=0$. Hence, by applying $S^1$-valued Livshits Theorem, there exists a function $w:M\to S^1$, which is differentiable along the flow direction, such that 
	 $$
	 \Phi(x,T)=w(X^T(x))-w(x)
	 $$
	  for every $x\in M$ and $T\in\R$. By applying the preceding discussion to $w$ we have the decomposition $w=\bar u+v$, where $\omega=dv$ is smooth and $\bar u\colon M\to S^1$ lifts to a function $u\colon M\to\R$. Since $v$ is smooth, $\bar u$ and $u$ are also differentiable along the flow direction. So, taking the limit of 
	  $$\frac1T\Phi(x,T)=\frac1T(w(X^T(x))-w(x))=\frac1T(\bar u(X^T(x))-\bar u(x))+\frac1T(v(X^T(x))-v(x))$$
	   as $T\to 0$, we obtain that $\phi$ is an abelian coboundary
	  $$\phi=L_X\bar u+\omega(X)=L_Xu+\omega(X)$$	
\end{proof}

\begin{remark} If $\phi=1$, \ie the flow $X^t$ only has  integer length periodic orbits then we have $\omega(X)=1$. Therefore by Schwartzman's theorem~\cite{Sch} the flow is a suspension. Further if we denote  by $S\subset M$ the section then for any periodic orbit $\gamma$ we have
$$
\ell(\gamma)=\int_\gamma\omega(X)dt=\langle \omega, \gamma\rangle =c\cdot \#\{S\cap \gamma\}
$$
Hence, we can apply the Livshits Theorem to the roof function and obtain that the roof function is smoothly cohomologous to a constant $c$, \ie the flow $X^t$ is a constant roof suspension over an Anosov diffeomorphism.
\end{remark}

\begin{question} In the setting of Proposition~\ref{proposition:periods} assume that $\scP$ has finite rank instead of rank one. Does there exist $\omega\in H^1(M,\R)$ and a H\"older continuous function $u$ such that
	$$
	\phi=\omega(X)+L_Xu?
	$$
\end{question}

\section{Reparametrizations of flows}
\label{sec_rep}

In this section we introduce some preliminaries on reparametrized flows (see Parry~\cite{P} for a more detailed introduction).
We study how equilibrium states change under reparamerization. We also examine the behavior of Sharp's minimizer under reparametrization.

 Given a $X^t$-invariant measure $\mu$ we will denote by $h_\mu(X)$ the {\it metric entropy} of $X^t$.
 Also recall that given a flow $X^t\colon M\to M$ and a H\"older continuous function $\phi\colon M\to\R$ the {\it pressure} $P_X(\phi)$ is defined by 
$$
P_X(\phi) =\sup\left\{ {h_\nu(X) + \int \phi d\nu : \nu\,\,\, \mbox{invariant probability measure}}\right\}
$$
When $X^t$ is a transitive Anosov flow, the unique measure $\nu_\phi=\nu_{\phi,X}$ realizing the supremum is called the {\it equilibrium state} of $\phi$ with respect to $X^t$.

Let $X^t\colon M\to M$ be a flow generated by the vector field $X$, let $\ell\colon M\to \R$ be a positive function. Define $Z=\ell X$ to be the generator of the {\it reparametrized flow} $Z^t$. Then $Z^t=X^{\tau_t}$. Here $\tau_t\colon M\to\R$ is the $Z^t$-cocycle with infinitesimal generator $\ell$  and is given by
$$
\tau_t(x)=\int_0^t\ell(Z^s(x))ds
$$ 
Similarly, if $k=1/\ell$, then $X^t=Z^{\kappa_t}$, where $\kappa_t\colon M\to\R$ is a $X^t$-cocycle given by
$$
\kappa_t(x)=\int_0^t k(X^s(x))ds
$$ 
\begin{lemma}\label{lemma_repar}
A reparametrization $Z^t$ as above is conjugate to $X^t$ via a time$-u$ map $X^u\colon M\to M$, $u\colon M\to\R$, if and only if
$$
L_Xu=\frac1\ell-1
$$
Further, two reperamerizations, $Z_i=\ell_iX$, $i=1,2$, are mutually conjugate via $X^u$ if and only if
$$
L_Xu=\frac1{\ell_1}-\frac1{\ell_2}
$$
\end{lemma}
\begin{proof} The conjugacy relation $X^u\circ Z^t=X^t\circ X^u$ yields
$$
u(Z^tx)+\tau_t(x)=t+u(x)
$$
or
$$
u(Z^tx)-u(x)=t-\tau_t(x)
$$
Dividing by $t$ and taking the limit as $t\to 0$ gives $L_Zu=1-\ell$. It remains to notice that $L_Zu=\ell L_Xu$. To show that $X^u$ is a conjugacy when $L_Xu=1/\ell-1$ one works backwards to obtain  $X^u\circ Z^t=X^t\circ X^u$  by integrating.

For the last statement notice that $Z_1=\frac{\ell_1}{\ell_2}Z_2$ and apply the criterion.
\end{proof}

Also recall the following result of Anosov and Sinai.
\begin{proposition}\label{prop3.2} If $X^t\colon M\to M$ is an Anosov flow and $Z^t$ is a smooth reparametrization of $X^t$, that is $Z=\ell X$, where $\ell$ is positive and smooth then $Z^t$ is also Anosov.
\end{proposition}

If $\mu$ is a $X$-invariant measure then 
$$
\mu_k=\frac{k}{\int k d\mu}\mu
$$ 
is $Z^t$-invariant. Recall that by definition, entropy of a flow is the entropy of its time-1 map. Then the Abramov entropy formula gives
$$
h_{\mu_k}(Z)=\frac{h_\mu(X)}{\int k d\mu}
$$

\begin{proposition}\label{pressureforreparametrization}
	Let $X^t\colon M\to M$ be a  transitive Anosov flow and let $Z=\ell X$, $\ell>0$, be a smooth reparametrization. Let $\phi\colon M\to\R$ be a H\"older continuous function and let $k={1}/{\ell}$. Then 
	$$
	P_Z(\ell(\phi-P_X(\phi)))=0
	$$ 
	Moreover, if $\nu_{\phi,X}$ is the equilibrium measure for $\phi$ with respect to  $X^t$, then the equilibrium measure $\nu_{\ell(\phi-P_X(\phi)),Z}$ for $\ell(\phi-P_X(\phi))$ with respect to  $Z^t$ is given by 
	$$
	\nu_{\ell(\phi-P_X(\phi)),Z}=\frac{k}{\int k \nu_{\phi,X}}\nu_{\phi,X}
	$$
\end{proposition}

\begin{proof}

	Take any $X$-invariant measure $\nu$ and let $\hat\nu=\frac{k}{\int kd\nu}\nu$. Then
	
	\begin{eqnarray*}
	h_{\hat\nu}(Z)+\int\ell(\phi-P_X(\phi))d\hat\nu&=&\frac{1}{\int kd\nu}\left(h_\nu(X)+\int(\phi-P_X(\phi))d\nu\right)\\
	&=&\frac{1}{\int kd\nu}\left(\left(h_\nu(X)+\int\phi d\nu\right)-P_X(\phi)\right)\leq 0
	\end{eqnarray*}
Further, the equality in the above inequality holds if and only if $\nu=\nu_{\phi,X}$. Since the correspondence $\nu\mapsto\hat\nu$ is a one-to-one and onto correspondence between invariant measures for $X$ and $Z$, uniqueness of equilibrium measures yields the posited result.
\end{proof}

%Given a homologically full Anosov flow $X^t\colon M\to M$ we follow~\cite{Sh} to define $\beta\colon H^1(M,\R)\to\R$ using the formula 
%$$
%\beta([\theta])=P_X(\theta(X))=\sup_\mu\left\{h_\mu(X)+\int\theta(X)d\mu \right\}\footnote{\label{note1} The definition is independent of a particular 1-form representative in the cohomology class because pressure and the equilibrium state only depend on the $X^t$-cohomology class of the potential.}
%$$ where the supremum is taken over all $X^t-$invariant probability measures $\mu$.

   %\marginpar{F: can we say we abuse notation and use $\omega$ for closed forms and for its cohomology class whenever does not lead to confusion}
%By~\cite[Theorem 1]{Sh}, functional $\beta$ is strictly convex, moreover, there exists a unique cohomology class
%$\xi_X\in H^1(M,\R)$ such that
%$$
%\beta(\xi_X) =  \inf   \{ {\beta([\theta]) : [\theta] \in H^1(M,\R)}\} 
%$$

Now let $X^t\colon M\to M$ be a homologically full Anosov flow. Recall the definition of $\beta$-functional, $\beta([\theta])=P_X(\theta(X))$, and Sharp's minimizer $\xi_X$ from Section~2.
Denote by $\mu_\xi$ the equilibrium state of $\xi_X(X)$.  
Also recall that by Step~1 of the proof of~\cite[Theorem 1]{Sh}
\begin{equation}
\label{eq_form_zero}
\int_M\theta(X)d\mu_\xi=0
\end{equation}
 for any closed $1-$form $\theta$. It follows that 
 \begin{equation}
 \label{eq_pressure}
 P_X(\xi_X(X))=h_{\mu_\xi}(X)
 \end{equation}

\begin{proposition}\label{repbyform}
	Let $X^t$ be a homologically full Anosov flow and let $Z^t$ be a reparametrization given by $Z=\ell X$ where $\ell={1}/({a+\omega(X)})$ for some positive constant $a$ and some smooth closed $1-$form $\omega$  with $a+\omega(X)>0$. Then 
	$$
	\xi_Z=\xi_X+\frac{P_X(\xi_X(X))}{a}[\omega]
	$$
\end{proposition}

\begin{proof}
	By Proposition \ref{pressureforreparametrization} we have $P_Z(\ell(\theta(X)-P_X(\theta(X)))=0$ for every $1-$form $\theta$. Applying $\omega$ to  
	$$
	Z=\frac{X}{a+\omega(X)}
	$$
	yields 
	$$
	\omega(Z)=\frac{\omega(X)}{a+\omega(X)}
	$$
	 and hence 
	 $$
	 \ell=\frac{1}{a+\omega(X)}=\frac{1}{a}(1-\omega(Z))
	 $$ 
	Then
	$$
	\ell(\theta(X)-P_X(\theta(X)))=\theta(Z)+\frac{\omega(Z)}{a}P_X(\theta(X))-\frac{1}{a}P_X(\theta(X))
$$
and we obtain	
	$$
	P_Z\left(\theta(Z)+\frac{\omega(Z)}{a}P_X(\theta(X))-\frac{1}{a}P_X(\theta(X))\right)=0
	$$ 
	or
	$$
	P_Z\left(\theta(Z)+\frac{\omega(Z)}{a}P_X(\theta(X))\right)=\frac{1}{a}P_X(\theta(X))
	$$ 
	 for every $\theta$. 
	
	We use the above formula to conclude that the map  
	$$\theta\mapsto \theta+\frac{P_X(\theta(X))}{a}\omega$$
	 is an invertible bijection on cohomology $H^1(M,\R)$. Indeed the inverse is given by $\theta\mapsto \theta-P_Z(\theta(Z))\omega$.  Hence, if $\theta$ minimizes 
	$$
	\theta\mapsto P_X(\theta(X))
	$$ 
	then 
	$$
	\hat\theta=\theta+\frac{P_X(\theta(X))}{a}\omega
	$$
	 minimizes $P_Z(\hat\theta(Z))$. Hence, by uniqueness of Sharp's minimizer.
	$$
	\xi_Z=\xi_X+\frac{P_X(\xi_X(X))}{a}[\omega]
	$$ 
\end{proof}

Finally we use Proposition~\ref{pressureforreparametrization} to show that any homologically full Anosov flow can be reparametrized so that Sharp's minimizer becomes zero. Sharp proved that such flows are special in the sense that periodic orbit growth is balanced in different homology classes~\cite[Section 5]{Sh}. We will need the following lemma.

\begin{lemma}
\label{lemma_stand}
Assume that for a cohomology class $\mu$ and each $X^t$-invariant probability measure $\nu$ we have $\int_M\mu(X)d\nu>-1$. Then $\mu$ can be represented by 1-form $\omega$ such that $\omega(X)>-1$. 
\end{lemma}
The proof is very standard  and is similar to the proof of uniform convergence of Birkhoff ergodic averages for uniquely ergodic systems. We just indicate the approach.

Fix a closed 1-form $\omega^0$ which represents $\mu$.
For any $\lambda >0$ let
$$
\omega^\lambda_x=\frac1\lambda\int_x^{X^\lambda(x)}(X^t)^*\omega^0_{X^t(x)}dt
$$
We have $[\omega^\lambda]=\mu$ and $\omega_x^\lambda(X)$ is given by the ergodic average of $\omega^0(X)$ 
$$\omega_x^\lambda(X)=\frac1\lambda\int_x^{X^\lambda(x)}\omega^0_{X^t(x)}(X(X^t(x))dt$$
Then the condition on the integrals of $\omega(X)$ implies that 
$\omega^\lambda(X)>-1$
for a sufficiently large $\lambda$.

\begin{proposition}
\label{prop_minimizer}
Let $X^t\colon M\to M$ be a homologically full Anosov flow. Then there is a unique (up to conjugacy) reparametrization of the form
$$
Z=\frac{X}{1+\omega(X)}
$$
which has zero Sharp's minimizer. Here $\omega$ is a closed $1$-form.
\end{proposition}

\begin{proof}
By Proposition~\ref{pressureforreparametrization} we have $\xi_Z=\xi_X+P_X(\xi_X(X))[\omega]$ and hence cohomology class $[\omega]$ is uniquely determined
$$
[\omega]=\frac{-1}{P_X(\xi_X(X))}\xi_X
$$
(Recall that by~(\ref{eq_pressure}) $P_X(\xi_X(X))=h_{\mu_\xi}(X)>0$.)

Thus it remains to show that the class $[\omega]$ can be realized by a 1-form $\omega$ with $\omega(X)>-1$. According to Lemma~\ref{lemma_stand} it is enough to show that
$$
\int_M\omega(X)d\nu>-1
$$
for every $X^t$-invariant probability measure $\nu$. If $\nu=\mu_\xi$ then 
$$
\int_M\omega(X)d\nu=0
$$
by~(\ref{eq_form_zero}).  Otherwise, if  $\nu\neq\mu_\xi$ then
$$
\int_M\omega(X)d\nu=\frac{-1}{P_X(\xi_X(X))}\int_M\xi_X(X)d\nu>\frac{-\int_M\xi_X(X)d\nu}{h_{\nu}(X)+\int_M\xi_X(X)d\nu}\ge -1
$$
\end{proof}

%\marginpar{maybe needed. if yes, maybe move to the place needed}
%\begin{proposition}[\cite{BR}]
%	Given a transitive Anosov flow $X^t$ and H\"older continuous functions $\phi$ and $\psi$  the following are equivalent
%	\begin{enumerate}
%		\item there is a constant $c\in\R$ and a H\"older continuous function $u$ which is differentiable along $X$ direction such that $L_Xu=\phi-\psi+c$;
%		\item there is $c\in\R$ such that $\int\phi d\delta_\gamma-\int\psi d\delta_\gamma=c$ for every periodic orbit $\gamma$;
%		\item the equilibrium states coincide $\nu_\phi=\nu_\psi$.
%	\end{enumerate}
%\end{proposition}

%================================RIDITY OF FULL FLOWS================================

\section{Conjugacy for homologically full flows.}\label{sectionfullflows}

Recall that a transitive Anosov flow $X^t\colon M\to M$ is {\it homologically full} if every integral homology class contains a periodic orbit of $X^t$. Notice that being homologically full is a property which is invariant under any orbit equivalence.

Two flows $X_i^t\colon M\to M$, $i=1,2$ are {\it conjugate} if there exists a homeomorphism $H\colon M\to M$ such that
$$
\forall t\,\,\,\,\,\,\, H\circ X^t=Y^t\circ H
$$
We say that $X_1^t$ and $X_2^t$ are {\it orbit equivalent} if there exists a homeomorphism $H$ which send orbits of $X_1^t$ to to orbits of $X_2^t$ preserving the time direction.

Let $X_i^t\colon M\to M$, $i=1,2$, be orbit equivalent Anosov flows. Fix an orbit equivalence $H_0\colon M\to M$ which sends orbits of $X_1^t$ to orbits of $X_2^t$. We say that $H_0$ {\it matches period spectra} if for every periodic point $x$ the $X_1^t$-period of $x$ is the same as $X_2^t$-period of $H_0(x)$. And we say that $H_0$ {\it matches homologically trivial period spectra} if only the periods of corresponding homologically trivial periodic orbits are assumed to be the same. Note that matching is not merely a property $X_1^t$ and $X_2^t$, but also depends on the choice of $H_0$ because flows can admit multiple non-equivalent orbit equivalences.

Recall the following classical application of the Livshits Theorem due to Katok.

\begin{theorem}\label{rigidityflows}
Let $X_1^t$ and $X_2^t$ be transitive Anosov flows and let $H_0$ be an orbit equivalence which matches period spectra. Then $X_1^t$ and $X_2^t$ are conjugate
$$
H\circ X_1^t=X_2^t \circ H
$$
where $H\colon M\to M$ is a bi-H\"older continuous homeomorphism.
\end{theorem}

\begin{proof}
The orbit equivalence $H_0$ is a bi-H\"older homeomorphism. By adjusting in the time direction we can also make $H_0$ continuously differentiable in the flow direction. Define
$$
Z^t=H_0\circ X_1^t\circ H_0^{-1}
$$
Then $Z^t$ is a H\"older continuous reparamerization of $X_2^t$ with the same periods. 

We use the same notation as in Section~\ref{sec_rep}: $Z=\ell X_2$, $Z^t=X_2^{\tau_t}$, $X_2^t=Z^{\kappa_t}$.  For any periodic point $x$ of period $T$ we have
$$
x=Z^T(x)=X_2^T(x)=Z^{\kappa_T}(x)
$$
and, hence, $\kappa_T(x)=T$ or
$$
\int_0^T\left(\frac1{\ell(X_2^s(x))}-1\right) ds=0
$$
Now, by Livshits Theorem there exists a H\"older function $u\colon M\to\R$ which is continuously differentiable along $X_2$ such that $L_{X_2}u=\frac1\ell-1$.
We conclude from Lemma~\ref{lemma_repar} that $Z^t$ and $X_2^t$ are bi-H\"older conjugate.
\end{proof}

Recall from Section~3 that $\xi_X\in H^1(M,\R)$ denotes the Sharp's minimizer for a homologically full Anosov flow $X^t\colon M\to M$.
\begin{theorem}\label{rigidityforfullflows}
Let $X_1^t\colon M\to M$ be a homologically full Anosov flow. Assume that  $X_2^t\colon M\to M$ is another Anosov flow which is orbit equivalent to $X_1^t$ via $H_0$. Assume that $H_0$ matches homologically trivial period spectra. Then  $X_1^t$ is conjugate to the reparametrization of $X_2^t$ generated by
$$
\frac{X_2}{1+\omega(X_2)}
$$
 where $\omega$ is a smooth closed $1-$form. If, moreover,  $H_0^*\xi_{X_2}=\xi_{X_1}$ then $\omega$ can be chosen to be zero and, hence, $X_1^t$ and $X_2^t$ are conjugate.
\end{theorem}

%\begin{corollary}\label{cor_rigidityforfullflows}
%Let $X_1^t$, $X_2^t$ and $H_0$ be as in Theorem~\ref{rigidityforfullflows}. Assume additionally that $\dim M=3$ or ?????? Then  $X_1^t$ is smoothly conjugate to the reparametrization of $X_2^t$ generated by
%$$
%\frac{X_2}{1+\omega(X_2)}
%$$
% where $\omega$ is a smooth closed $1-$form. If, moreover,  $H_0^*\xi_{X_2}=\xi_{X_1}$ then $\omega$ can be chosen to be zero and, hence, $X_1^t$ and $X_2^t$ are smoothly conjugate.
%\end{corollary}

\begin{proof}
The proof proceeds in exactly the same way as the proof of Theorem~\ref{rigidityflows}, but instead of applying the classical Livshits Theorem we apply Theorem~\ref{livsic} and obtain a function $u\colon M\to\R$ and a smooth closed $1-$form $\omega$ such that
$$
L_{X_2}u=\frac1\ell-1-\omega(X_2)
$$
Note that we can approximate $L_{X_2}u$ with $L_{X_2}u'$, where $u'$ is smooth. Then, after replacing $\omega$ with $\omega-du'$ we have (cf. Remark~\ref{rmk_liv})
$$
L_{X_2}(u-u')=\frac1\ell-1-\omega(X_2)
$$
and for a sufficiently small $L_{X_2}(u-u')$ we have $1+\omega(X_2)>0$.
 Then by Lemma~\ref{lemma_repar} flow $Z^t$ (and hence $X_1^t$) is conjugate to the reparametrization of $X_2^t$ generated by 
$X_2/(1+\omega(X_2))$. This gives us the first part of the theorem. 

Hence without loss of generality we can (and do) assume that
 $$
 Z=\frac{1}{1+\omega(X_2)}X_2,
 $$ 
%We let $\ell=\frac{1}{1+\omega(X_2)}X_2,$ and $k=1+\omega(X_2)$.

It is left to check that $\omega$ is exact if the cohomology classes $\xi_{X_1}$  and $\xi_{X_2}$ match. 

Because $Z^t$ is conjugate to $X_1^t$ we have $\xi_{X_1}=H_0^*\xi_Z$ and hence, by the assumption of the theorem, $\xi_Z=\xi_{X_2}$.
Recall that  by~(\ref{eq_pressure}) we have $P_{X_2}(\xi_{X_2}(X_2))=h_{\mu_{X_2}}(X_2)>0$. Thus applying Proposition~\ref{repbyform} with $a=1$ and $X=X_2$ we obtain that $[\omega]=0$, that is, $\omega$ is exact. Hence, by Lemma~\ref{lemma_repar}, flows $Z^t$ and $X_2^t$ are conjugate.
\end{proof}

%=================================SHARPENING MLS==============================================

\section{Sharpened Marked Length Spectrum Rigidity}\label{sharpotal}

Here we explain that our abelian Livshits theory can be used to improve marked length spectrum rigidity results on surfaces and higher dimensional manifolds. Recall that Croke~\cite{Cr} and Otal~\cite{Ot} famously proved that marked lengths of closed geodesics determine the isometry class of a negatively curved surface. We offer the following enhancement. 

Given a negatively curved surface $(S,g)$, a free homotopy class of loops $\alpha$ on $S$ admits a unique geodesic representative. Denote by $\ell_g(\alpha)$ the length of this geodesic.
\begin{theorem}
\label{thm_sharpenedMLS}
	Let $g_1$ and $g_2$ be two negatively curved metrics on a smooth compact surface $S$. Given a free homotopy class of loops $\alpha$ on $S$, denote by $\ell_{g_i}(\alpha)$ the length of the $g_i$-geodesic representative of $\alpha$, $i=1,2$. Assume that their marked length spectra are the same for homologically trivial geodesics, \ie  $\ell_{g_1}(\alpha)=\ell_{g_2}(\alpha)$ for every homologically trivial free homotopy class of loops $\alpha$. Then $g_1$ and $g_2$ are isometric. 
\end{theorem}

In fact the following more general result holds true.
\begin{add}
Fix a homology class $c\in H_1(S,Z)$. If instead we assume that $\ell_{g_1}(\alpha)=\ell_{g_2}(\alpha)$ for every free homotopy class of loops in homology class $c$ then $g_1$ and $g_2$ are isometric. 
\end{add}

\begin{remark} A precursor for the idea of considering a fixed homology class can be found in~\cite[Theorem 3]{K2}, where  Katok proved that marked length spectrum in a fixed homology class $c$ determines the negatively curved metric on the surface in a fixed conformal class.
\end{remark}

We proceed to prove Theorem~\ref{thm_sharpenedMLS} below. The addendum can be reduced to Theorem~\ref{thm_sharpenedMLS} in the following way. The length of a homologically trivial geodesic can be arbitrarily well approximated by the difference of lengths of two geodesics in homology class $c$. This approximation can be done in the same way as in the proof of Lemma~\ref{lemma_dense}. Hence the length of such homologically trivial geodesic can be recovered from the lengths of geodesics in $c$. Further, the approximation procedure persists under orbit equivalence of geodesic flows and, hence, marked length spectrum in $c$ recovers the homologically trivial length spectrum.

\begin{lemma}\label{vanishingMME} Let $X^t\colon T^1S\to T^1S$ be the geodesic flow on a negatively curved surface and let $I\colon T^1S\to T^1S$ be the involution given by $v\mapsto -v$. Assume that $\mu$ is an $X^t$-invariant measure such that $I_*\mu=\mu$. Then  
 $$
 \int \omega(X)d\mu=0
 $$ 
 for every closed $1-$form  $\omega$. In particular, this holds for the measure of maximal entropy.
\end{lemma}
\begin{proof} Note that $I(x,v)=(x,-v)$ conjugates the geodesic flow and it's inverse and interchanges the stable and unstable foliations. We have $X^{-t}=IX^tI$ or, infinitesimally, $DI(X)=-X$. Hence, if $\mu$ is the measure of maximal entropy for $X^t$ then $I_*\mu$ is the measure of maximal entropy for $X^{-t}$ and, hence, indeed, $I_*\mu=\mu$. 

Recall that the bundle map $T^1S\to S$ induces an isomorphism on cohomology. Hence we can and do assume that $\omega$ is a pullback of a $1-$form on the surface. It is easy to see that for such forms we have $\omega_{(x,-v)}(X(x,-v))=-\omega_{(x,v)}(X(x,v))$.
Now, for any $\mu$ such that $I_*\mu=\mu$, the claim of the lemma comes from the following calculation.
\begin{multline*}
 \int \omega(X)d\mu=	 \int \omega(X)dI_*\mu=\int\omega_{(x,-v)}(X(x,-v))d\mu(x,v)\\
 =\int-\omega_{(x,v)}(X(x,v))d\mu(x,v)=-\int\omega(X)d\mu.
\end{multline*}
	
\end{proof}

\begin{proof}[Proof of Theorem~\ref{thm_sharpenedMLS}]
Denote by $X_1$ and $X_2$ the generating vector fields of geodesic flows on $T^1S$ corresponding to $g_1$ and $g_2$, respectively. Then, it is well-known that there is exists  $H\colon T^1S\to T^1S$, an orbit equivalence between $X_1^t$ and $X_2^t$, which is homotopic to identity. 

Recall the definition of $\beta$-functional
$$
\beta([\theta])=P(\theta(X))=\sup_\mu\left\{h_\mu(X)+\int\theta(X)d\mu\right\}
$$
where the supremum is taken among all $X^t-$invariant probability measures.

Denote by $\xi_i$, $i=1,2$, the Sharp's minimizers of the $\beta$ functional for $X_i$. Let $\mu_{\xi_i}$ be the equilibrium measures for $\xi_i(X)$ and also denote by $\mu_i$, $i=1,2$, the measures of maximal entropy for $X_i$. 

Then, using~(\ref{eq_form_zero}) which gives $\int\xi_i(X_i)d\mu_{\xi_i}=0$ and Lemma~\ref{vanishingMME}, we have
$$
\beta(\xi_i)=h_{\mu_{\xi_i}}(X_i)+\int\xi_i(X_i)d\mu_{\xi_i}=h_{\mu_{\xi_i}}(X_i)\le h_{\mu_{i}}(X_i)=h_{\mu_{i}}(X_i)+\int\xi_i(X_i)d\mu_{i}
$$
Hence, by the definition of $\beta$, we have $\mu_{\xi_i}=\mu_i$ and, by uniqueness of Sharp's minimizer, $\xi_i=0$, $i=1,2$. Then, obviously, $H^*\xi_2=0=\xi_1$ and we can apply Theorem~\ref{rigidityforfullflows} to conclude that $X_1^t$ and $X_2^t$ are conjugate. Hence, we have complete matching of marked length spectra and, by Croke-Otal rigidity theorem, $g_1$ is isometric to $g_2$.
\end{proof}

\begin{remark}
Recent results of Guillarmou and Lefeuvre on local marked length spectrum rigidity~\cite[Theorem 1]{GL} for higher dimensional negatively curved manifolds can be enhanced in the same way --- one only needs to assume that homologically trivial marked length spectra coincide. 
\end{remark}

%=============================CONTACT ANOSOV FLOWS====================================

\section{Conjugacy for contact Anosov flows}\label{contact}

	Recall that a (positive) {\it contact form} on an oriented $(2k+1)$-dimensional manifold $M$ is a smooth $1-$form $\alpha$ such that $\alpha\wedge (d\alpha)^k>0$. Associated to the contact form is its {\it Reeb vector field}  $X_\alpha$ which is uniquely determined by $\alpha(X_\alpha)=1$ and $\cL_{X_\alpha}\alpha=0$ (the latter is equivalent to $\iota_{X_\alpha}d\alpha=0)$. Call an Anosov flow $X^t$ a {\it contact Anosov flow} if $X$ is the Reeb vector field for a contact form $\alpha$.

%\begin{lemma}
%	If $\phi X=Y$ then $Y$ is time preserving conjugate to $X$ iff there is $u$ s.t. $du(Y)=\phi+c$ for some constant $c$.
%\end{lemma}

\begin{theorem}
\label{thm10.3}
	Let $X_i$, $i=1,2$, be contact Anosov flows. Assume one of the following
	\begin{enumerate}
	\item $X_i$ are flows on a 3-dimensional manifold, which are orbit equivalent via a $C^1$ orbit equivalence;
	\item $X_i$ are Anosov geodesic flows with $C^1$ Anosov splittings, which are orbit equivalent via a $C^2$ orbit equivalence;
	\end{enumerate}
	Then there exist a closed $1-$form $\omega$ and a constant $c>0$ such that $X_1$ is smoothly conjugated to $\frac{X_2}{c+\omega(X_2)}$. If, moreover, the orbit equivalence matches Sharp's minimizers $\xi_{X_1}\in H^1(M,\R)$ to $\xi_{X_2}\in H^1(M,\R)$, then $\omega$ can be taken to be $0$, that is, $X_1$ is conjugate to a constant rescaling of $X_2$. 
\end{theorem}
Recall that geodesic flows on perturbations of hyperbolic manifolds are $\frac12$-pinched and hence have $C^1$ Anosov splittings. 

\begin{add} In the first case when $X_i$ are 3-dimensional flows the $C^0$ conjugacy is in fact smooth.
\end{add}

The addendum follows from work of Fledman-Ornstein~\cite{FO} who proved that a $C^0$ conjugacy must be $C^1$ and the bootstrap argument of de la Llave-Moriy\'on~\cite{dlLM}.

\begin{proof}[Proof of Theorem~\ref{thm10.3}]
Let $\alpha$ be the contact form for $X_1$ and let $\beta$ be the contact form for $X_2$. Denote by $H$ the orbit equivalence so that $H_*X_1=\phi X_2$ for some positive $\phi\in C^0$.
\begin{lemma}
There exists a constant $c>0$ such that the 1-form $\mu=c\beta-H_*\alpha$ is closed.
\end{lemma}
\begin{proof}
First we prove the lemma when $X_i$ are Anosov geodesic flows with $C^1$ Anosov splitting and $H$ is $C^2$. Because of the $C^2$ hypotheis we have that $\phi\in C^1$ and $H_*d\alpha=dH_*\alpha$ is exact. We claim that it is also $X_2$-invariant. Using functoriality, we have
$$
0=L_{X_1}d\alpha=L_{\phi X_2}H_*d\alpha=\phi L_{X_2}(H_*d\alpha)+d\phi\wedge \iota_{X_2} H_*d\alpha
$$
Notice that $0=\iota_{X_1}d\alpha=\iota_{\phi X_2}H_*d\alpha=\phi\iota_{X_2} H_*d\alpha$. Hence, indeed, we have $\phi L_{X_2}(H_*d\alpha)=0$.

Therefore, both $d\beta$ and $H_*d\alpha$ are exact $X_2$-invariant 2-forms. Then by~\cite[Theorem A3]{Ham} there is a constant $c>0$ such that $H_*d\alpha=c d\beta$ and lemma follows. (Constant $c$ is positive because both $\alpha$ and $\beta$ are positive contact forms and $\phi>0$.)

In the 3-dimensional case when $H$ is merely $C^1$ (and, hence, we do not know that $H_*d\alpha$ is exact anymore) we can actually make a direct argument. We have that $H_*(\alpha\wedge d\alpha)$ is a $\phi X_2$-invariant $C^0$ volume form. Hence, both $\beta\wedge d\beta$ and $\phi H_*(\alpha\wedge d\alpha)$  are $X_2$-invariant volume forms. Then, by ergodicity, 
$$
\phi H_*(\alpha\wedge d\alpha)=c \beta\wedge d\beta
$$
where $c>0$, again, because both $\alpha$ and $\beta$ are positive. 

Now note that $\iota(\beta\wedge d\beta)=d\beta$ and $\iota(\alpha\wedge d\alpha)=d\alpha$. We calculate
\begin{eqnarray*}cd\beta&=&c \iota_{X_2}(\beta\wedge d\beta)=\frac{c}{\phi}\iota_{H_*X_1}(\beta\wedge d\beta)
	=\iota_{H_*X_1}\left(\frac{c}{\phi}\beta\wedge d\beta\right)\\
	&=&\iota_{H_*X_1}\left(H_*(\alpha\wedge d\alpha)\right)=H_*\iota_{X_1}\left(\alpha\wedge d\alpha\right)\\
	&=&H_*d\alpha
\end{eqnarray*}
\end{proof}

Consider any homologically trivial $X_1$-periodic orbit $\gamma$. Then $H_*\gamma$ bounds a surface $S$ and we have
\begin{multline*}
c\,\textup{per}_{X_2}(H_*\gamma)-\textup{per}_{X_1}(\gamma)=c\int_{H_*\gamma}\beta-\int_\gamma\alpha=\int_{H_*\gamma}c\beta-\int_{H_*\gamma}{H_*\alpha}\\
=\int_{H_*\gamma}\mu=\int_Sd\mu=0
\end{multline*}
Hence, after rescaling by $c$, the marked length spectra for homologically trivial orbits match. Hence the result follows immediately from Theorem~\ref{full} and Theorem~\ref{rigidityforfullflows}.
\end{proof}

\begin{remark}
For non-homotopically trivial periodic orbits $\gamma$ the above calculation gives that  the periods $\textup{per}_{X_1}(\gamma)$ and $\textup{per}_{X_2}(H_*\gamma)$ are related as follows
$$
\textup{per}_{X_1}(\gamma)= c\,\textup{per}_{X_2}(H_*\gamma)+[\mu]([H_*\gamma])
$$
Note that $[H_*\gamma]=[\gamma]$ if $H$ is homotopic to identity.
\end{remark}

\section{An example}

We begin by pointing out that the scenario of Theorem~\ref{thm10.3} actually occurs. Indeed, given a contact Anosov flow $X$ with a contact form $\beta$ and a closed 1-from $\omega$ with $\omega(X)>-1$, then the reparametrization
$$
X_\omega=\frac{X}{1+\omega(X)}
$$
is a Anosov by Proposition~\ref{prop3.2} and contact with contact form $\beta+\omega$. Note that Remark~6.4 implies $X^t$ is not conjugate to $X_\omega^t$ if $X^t$ admits a periodic orbit on which $[\omega]$ does not vanish. In particular, if $X^t$ admits at least one homologically non-trivial (in $H_1(M,\R)$ \ie non-torsion) periodic orbit then there exists a small cohomology class $[\omega]$ such that $X^t$ is not conjugate it $X^t_\omega$.

Hence we see that indeed, unlike in Theorem~\ref{thm_sharpenedMLS}, matching of homologically trivial length spectra does not imply conjugacy for contact Anosov flows. This observation shows that that the conclusion of Theorem~\ref{thm10.3} is optimal. Further by Proposition~\ref{prop_minimizer} we can find the ``best" contact reparametrization with Sharp's minimizer equal to zero.

In this section we would like to present the same example from the point of view of deforming the Deck group rather than reparametrizing. At the end we will find out that this is  exactly the same example. While there is some redundancy with what was already discussed, we consider this alternative description quite instructive and thus give a rather detailed and self-contained presentation. 

 We will describe an explicit deformation $X_\mu^t\colon M_\mu\to M_\mu$ of a flow $X^t=X_0^t$ for small $\mu\in H^1(M,\R)$, such that the length spectrum deforms according to $\mu$
$$\textup{per}_{X_\mu}(h(p))= \textup{per}_{X}(p)+\mu(\gamma)$$
where $h$ is an orbit equivalence and $\gamma$ is the homology class of closed orbit of $x$.
The result of the construction will be summarized below as Proposition~\ref{last_prop}. In the context of geodesic flows on surfaces of constant negative curvature the same example was given by Ghys~\cite[Theorem 2.2]{G}. He was interested in examples of Anosov flows which are not conjugate to algebraic flows and have analytic stable and unstable distributions. We give a different, more general construction.

We proceed with the description of the example. Let $X^t\colon M\to M$ be a flow. We lift the flow to the universal abelian cover $\tilde X^t\colon\tilde M\to\tilde M$. 
%Also fix a Riemannian metric on $M$ and lift it to $\tilde M$ so that the flow has unit speed. 
Then $\Gamma_0\simeq H_1(M,\Z)$ is the group of Deck transformations acting by isometries on $\tilde M$. Note that $\gamma\in\Gamma_0$ commutes with $\tilde X^t$. 

Take a $\mu\in H^1(M,\R)\simeq \textup{Hom}(H_1(M,\Z),\R)$. Let
$$
\Gamma_\mu=\{\gamma\circ\tilde X^{\mu(\gamma)}:\gamma\in\Gamma_0\}
$$
It is easy to see now that $I_\mu\colon \Gamma_0\to\Gamma_\mu=Im(I_\mu)$ given by 
$$
I_\mu(\gamma)=\gamma\circ\tilde X^{\mu(\gamma)}
$$
is a group homomorphism (which is one-to-one when $X^t$ is not a periodic flow).

Then $\Gamma_\mu$ acts on $\tilde M$ by $\gamma\colon x\mapsto\gamma(\tilde X^{\mu(\gamma)}(x))$. First we will see that for all sufficiently small $\mu$ the orbit space $M_\mu=\tilde M/\Gamma_\mu$ is a smooth manifold diffeomorphic to $M$.

Recall that the isomorphism $H^1(M,\R)\simeq \textup{Hom}(H_1(M,\Z),\R)$ arises as follows. Let $\omega$ be a closed 1-form with cohomology class $[\omega]=\mu$. Then the lift $\tilde\omega$ of $\omega$ to $\tilde M$ is exact, that is, there exists a function $\alpha\colon\tilde M\to\R$ such that $\tilde \omega=d\alpha$. Then the homomorphism $\mu\colon H_1(M,\Z)\to \R$ is given by
$$
\mu(\gamma)=\int_x^{\gamma(x)}\tilde \omega=\alpha(\gamma(x))-\alpha(x)
$$
for any $x\in \tilde M$.

Now define $H\colon \tilde M\to\tilde M$ by $H(x)=\tilde X^{\alpha(x)}(x)$. Clearly $H$ sends orbits of $\tilde X^t$ to themselves. If
\begin{equation}
\label{alpha}
L_{\tilde X}\alpha>-1
\end{equation}
then $H$ is invertible on every orbit and, hence, is a smooth diffeomorphism. Further $H$ intertwines actions of $\Gamma_0$ and $\Gamma_\mu$
$$\forall \gamma\in\Gamma_0 \,\,\,\, H\circ \gamma=I_{\mu}(\gamma)\circ H$$ 
Indeed,
\begin{multline*}
 X^{\alpha(\gamma(x))}(\gamma(x))=X^{\alpha(x)+\mu(\gamma)}(\gamma(x))\\
 =X^{\alpha(x)}(X^{\mu(\gamma)}(\gamma(x)))=\gamma\circ X^{\mu(\gamma)}(X^{\alpha(x)}(x))=I_\mu(\gamma)\circ H
 \end{multline*}
 Hence, under condition~(\ref{alpha}), $H$ induces a diffeomorphism $h\colon M\to M_\mu$. Moreover, if we denote by $X_\mu^t\colon M_\mu\to M_\mu$ the flow induced by $\tilde X^t$ then $h$ is an orbit equivalence between $X^t$ and $X_\mu^t$.

\begin{remark} Note that $X_\mu=(L_X\alpha+1)Dh(X)$. Hence by Proposition~\ref{prop3.2}, under the condition~(\ref{alpha}), if $X$ is Anosov then so is $X_\mu$. Note also that if $X$ is contact then so is $X_\mu$ because the action of $\Gamma_\mu$ preserves the contact form. In fact, if $\beta$ is the contact form for $X^t$ then $\beta+\mu$ is the contact form for $h^{-1}\circ X_\mu^t\circ h$.
\end{remark}

Let $x\in M$ be a periodic point of period $\textup{per}_X(x)$ and let $\tilde x\in\tilde M$ be a lift of $x$. Then $\tilde X^{\textup{per}_X(x)}=\gamma(\tilde x)$, where $\gamma$ is the homology class of the orbit of $x$. Then
$$
\tilde X^{\textup{per}_X(x)+\mu(\gamma)}=\tilde X^{\mu(\gamma)}(\gamma(\tilde x))=I_\mu(\gamma)(\tilde x)
$$
and hence
\begin{equation}
\label{eq_per}
\textup{per}_{X_\mu}(h(p))= \textup{per}_{X}(p)+\mu(\gamma)
\end{equation}

Finally we have the following lemma.
\begin{lemma}
\label{last_lemma}
The set
$$
\mathcal U=\{\mu\in H^1(M,\R): \exists\omega :  [\omega]=\mu, \,\|\omega\|_{C^0}<1\}
$$
is an open neighborhood of 0 in $H^1(M,\R)$. Further, if $\mu=[\omega]\in\mathcal U$ then condition~(\ref{alpha}) holds.
\end{lemma}

We summarize all of the above discussion as follows.

\begin{proposition} Given a smooth flow $X^t\colon M\to M$ on a compact manifold $M$. There exists a open neighborhood $\mathcal U\subset H^1(M,\R)$ of zero and a deformation $X_\mu^t\colon M_\mu\to M_\mu$ such that
\label{last_prop}
\begin{enumerate}
\item $X_0^t=X^t$;
\item There exists a family of diffeomorphisms $h_\mu\colon M\to M_\mu$ which give orbit equivalences between $X^t$ and $X^t_\mu$;
\item The periods of periodic orbits deform according to~(\ref{eq_per});
\item If $X^t$ is Anosov, then all $X_\mu^t$, $\mu\in\mathcal U$ are Anosov;
\item If $X^t$ is contact, then all $X_\mu^t$, $\mu\in\mathcal U$ are contact.
\end{enumerate}
\end{proposition}

\begin{remark}
While Lemma~\ref{last_lemma} is very simple and elementary, the actual description of the set of admissible cohomology classes $\mu$ appears in Lemma~\ref{lemma_stand}.
\end{remark}

It remains to prove the lemma.
\begin{proof}[Proof of Lemma~\ref{last_lemma}]
If $\mu=[\omega]$ with $\|\omega\|_{C^0}<1$ then we have $L_{\tilde X}\alpha=d\alpha(\tilde X)=\tilde\omega(\tilde X)=\omega(X)$. Hence if $\mu\in\mathcal U$ then condition~(\ref{alpha}) is verified. 

Now we check that $\mathcal U$ is open. Let $[\omega_1], [\omega_2],\ldots [\omega_N]$ be a basis of $H^1(M,\R)$. By rescaling if necessary, we can assume that $\|\omega_i\|_{C^0}=1$, $i=1,\ldots N$. Then, obviously, 
the set
$$
\cB_\eps=\left\{\frac1N\left[\sum_{i=1}^Nt_i\omega_i\right]: \sum_{i=1}^Nt_i\le\eps\right\}
$$
contains an open neighborhood of 0 in $H^1(M,\R)$ and any $\mu\in\cB_\eps$ can be represented by a closed 1-form of norm $\le\eps$. 

For any $\mu\in\mathcal U$ we have $\mu=[\omega]$ with $\|\omega\|_{C^0}<1$. Let $\eps=\frac12(1-\|\omega\|_{C^0})$. Then it is easy to see that $\mu+\cB_\eps\subset \mathcal U$ proving that $\mathcal U$ is open.
\end{proof}

\end{document}